\documentclass[a4paper,11pt]{amsart}
\title[Gepner point]{Gepner point and strong Bogomolov-Gieseker inequality 
for quintic 3-folds}
\date{}
\author{Yukinobu Toda}

\usepackage{amscd}
\usepackage{amsmath}
\usepackage{amssymb}
\usepackage{amsthm}
\usepackage{float}
\usepackage[dvips]{graphicx}

\usepackage[all,ps,dvips]{xy}

\usepackage{array}
\usepackage{amscd}
\usepackage[all]{xy}

\DeclareFontFamily{U}{rsfs}{%
\skewchar\font127}
\DeclareFontShape{U}{rsfs}{m}{n}{%
<-6>rsfs5<6-8.5>rsfs7<8.5->rsfs10}{}
\DeclareSymbolFont{rsfs}{U}{rsfs}{m}{n}
\DeclareSymbolFontAlphabet
{\mathrsfs}{rsfs}
\DeclareRobustCommand*\rsfs{%
\@fontswitch\relax\mathrsfs}

\theoremstyle{plain}
\newtheorem{thm}{Theorem}[section]
\newtheorem{prop}[thm]{Proposition}
\newtheorem{lem}[thm]{Lemma}

\newtheorem{defi}[thm]{Definition}
\newtheorem{rmk}[thm]{Remark}

\newtheorem{case}{Case}

\newtheorem{prop-defi}[thm]{Proposition-Definition}
\newtheorem{thm-defi}[thm]{Theorem-Definition}
\newtheorem{lem-defi}[thm]{Lemma-Definition}

\newtheorem{conj}[thm]{Conjecture}
\newtheorem{exam}[thm]{Example}

\newdimen\argwidth
\def\db[#1\db]{
 \setbox0=\hbox{$#1$}\argwidth=\wd0
 \setbox0=\hbox{$\left[\box0\right]$}
  \advance\argwidth by -\wd0
 \left[\kern.3\argwidth\box0 \kern.3\argwidth\right]}

\newcommand{\aA}{\mathcal{A}}
\newcommand{\bB}{\mathcal{B}}

\newcommand{\dD}{\mathcal{D}}
\newcommand{\eE}{\mathcal{E}}
\newcommand{\fF}{\mathcal{F}}

\newcommand{\hH}{\mathcal{H}}

\newcommand{\lL}{\mathcal{L}}
\newcommand{\mM}{\mathcal{M}}

\newcommand{\oO}{\mathcal{O}}
\newcommand{\pP}{\mathcal{P}}

\newcommand{\sS}{\mathcal{S}}
\newcommand{\tT}{\mathcal{T}}
\newcommand{\uU}{\mathcal{U}}

\newcommand{\Supp}{\mathop{\rm Supp}\nolimits}
\newcommand{\Hom}{\mathop{\rm Hom}\nolimits}

\newcommand{\dR}{\mathbf{R}}

\newcommand{\ch}{\mathop{\rm ch}\nolimits}

\newcommand{\td}{\mathop{\rm td}\nolimits}
\newcommand{\Ext}{\mathop{\rm Ext}\nolimits}

\newcommand{\rank}{\mathop{\rm rank}\nolimits}
\newcommand{\Coh}{\mathop{\rm Coh}\nolimits}

\newcommand{\cneq}{\mathrel{\raise.095ex\hbox{:}\mkern-4.2mu=}}
\newcommand{\eqcn}{\mathrel{=\mkern-4.5mu\raise.095ex\hbox{:}}}
\newcommand{\HMF}{\mathrm{HMF}^{\rm{gr}}}

\newcommand{\Aut}{\mathop{\rm Aut}\nolimits}

\newcommand{\Stab}{\mathop{\rm Stab}\nolimits}

\newcommand{\Imm}{\mathop{\rm Im}\nolimits}

\newcommand{\Ree}{\mathop{\rm Re}\nolimits}

\newcommand{\Auteq}{\mathop{\rm Auteq}\nolimits}

\begin{document}

\begin{abstract}
We propose a conjectural 
stronger version of Bogomolov-Gieseker 
inequality for stable sheaves on quintic 3-folds. 
Our conjecture is derived from an attempt
to construct a Bridgeland stability condition on 
graded matrix factorizations, which 
should correspond to the Gepner point 
via mirror symmetry and Orlov equivalence. 
We 
prove our conjecture in the 
rank two case. 
\end{abstract}

\maketitle

\setcounter{tocdepth}{1}
\tableofcontents

\section{Introduction}
\subsection{Bogomolov-Gieseker (BG) inequality}
First of all, let us recall the following 
classical result by Bogomolov and Gieseker:
\begin{thm}\emph{(\cite{Bog},~\cite{Gie})}
Let $X$ be a smooth projective complex variety and 
$H$ an ample divisor in $X$. 
For any torsion free $H$-slope 
stable sheaf $E$ on $X$, 
we have 
\begin{align*}
\Delta(E) \cdot H^{\dim X-2} \ge 0. 
\end{align*}
Here $\Delta(E)$ is the discriminant
\begin{align*}
\Delta(E) \cneq \ch_1(E)^2 - 2\ch_0(E) \ch_2(E).
\end{align*}
\end{thm}
It has been an interesting problem to improve 
the BG inequality for higher rank stable sheaves 
(cf.~\cite{Jar}, \cite{Naka}).
So far such an improvement is only known for 
some particular surfaces, e.g. K3 surfaces or 
Del Pezzo surfaces, 
which easily follows from Riemann-Roch theorem and 
Serre duality (cf.~Lemma~\ref{strong:K3}, \cite[Appendix~A]{DRY}).
In the 3-fold case,
such an improvement is only known for  
rank two stable sheaves on $\mathbb{P}^3$ by 
Hartshorne~\cite{Hart2}. In a case of other 3-fold, 
even a conjectural improvement is not known. 
The purpose of this note is to propose
a conjectural improvement of BG inequality for stable sheaves on 
quintic 3-folds, motivated by an idea from 
mirror symmetry and matrix factorizations. 
We first state the resulting conjecture:   
\begin{conj}\label{intro:strong}
Let $X \subset \mathbb{P}_{\mathbb{C}}^4$ be a smooth 
quintic 3-fold and $H\cneq c_1(\oO_X(1))$. 
Then for any torsion free $H$-slope stable 
sheaf on $X$ 
with $c_1(E)/ \rank(E)=-H/2$, 
we have the following inequality: 
\begin{align}\label{intro:bound}
\frac{\Delta(E) \cdot H}{\rank(E)^2} > 1.5139 \cdots. 
\end{align}
\end{conj}
The RHS of (\ref{intro:bound})
is a certain irrational real 
number contained in $\mathbb{Q}(e^{2\pi \sqrt{-1}/5})$, 
and the detail will be discussed in
 Conjecture~\ref{strong}.
Our conjecture is derived from an attempt to construct 
a Bridgeland stability condition
on $D^b \Coh(X)$ corresponding 
to the Gepner point in the stringy 
K\"ahler moduli space of $X$. 
The RHS of (\ref{intro:bound})
is related to the coefficient of 
the corresponding central charge. 
It seems that Conjecture~\ref{intro:strong}
does not appear in literatures even in 
the rank two case, which we will give a 
proof in this note: 
\begin{prop}
Conjecture~\ref{intro:strong}
is true if $\rank(E)=2$. 
\end{prop}
The above result will be proved in Subsection~\ref{rank:two}. 
Based on a similar idea, we also propose a 
conjectural Clifford type bound for
stable coherent systems on quintic 
surfaces (cf.~Section~\ref{sec:Cli}). 
Below we discuss background of the 
derivation of the above conjecture.

\subsection{Background}
The notion of stability conditions on 
triangulated categories introduced by Bridgeland~\cite{Brs1}
has turned out to be an important mathematical 
object to study. 
However 
it has been a problem for more than ten 
years to construct 
Bridgeland 
stability conditions on the derived 
categories of coherent sheaves on quintic 3-folds. 
From a picture of the mirror symmetry, 
the space of stability conditions on 
a quintic 3-fold is expected to be related to 
its stringy K\"ahler moduli space, which 
is described in Figure~\ref{fig:one}. 
In Figure~\ref{fig:one}, we see three 
special points, large volume limit, conifold point 
and Gepner point. 
A conjectural construction of a Bridgeland 
stability condition near the large volume limit
was proposed by Bayer, Macri and the author~\cite{BMT}, and 
we reduced the problem to showing a BG
type inequality evaluating $\ch_3(\ast)$ for certain two term 
complexes.
The main conjecture in~\cite{BMT}
is not yet proved except in the $\mathbb{P}^3$
case~\cite{MaBo}, and we face our lack of knowledge 
on the set of Chern characters of stable objects.

In this note, we focus on the Gepner point. 
A corresponding stability condition 
is presumably 
constructed as a Gepner type stability condition~\cite{TGep} 
with respect to the pair
\begin{align*}
\left( \mathrm{ST}_{\oO_X} \circ \otimes \oO_X(1), 
\frac{2}{5} \right)
\end{align*}
where $\mathrm{ST}_{\oO_X}$ is the 
Seidel-Thomas twist~\cite{ST} associated to $\oO_X$. 
Combined with Orlov's result~\cite{Orsin}, 
as discussed in~\cite{Wal}, 
such a stability condition is expected to give 
a natural stability condition on 
graded matrix factorizations of the defining polynomial of 
the quintic 3-fold. 
 One may expect that constructing a Gepner point 
also requires such a conjectural inequality. 
It seems worth formulating 
a conjectural BG type inequality which 
arises in an attempt to construct a Gepner point, 
so that making it clear what we should know on 
Chern characters of stable sheaves. 
Our Conjecture~\ref{intro:strong} is the 
resulting output. 
The inequality (\ref{intro:bound}) itself 
is interesting since there have been 
several attempts to improve the classical BG inequality.  
Assuming Conjecture~\ref{intro:bound}, we construct data
which presumably give a 
Bridgeland stability condition 
corresponding to the Gepner point. 

Compared to the lower degree 
cases studied in~\cite{TGep}, 
constructing Gepner type 
stability conditions is much 
harder in quintic cases, and 
most of the attempts 
are still conjectural. 
This is the reason 
we have separated the arguments for
the quintic case from the previous paper~\cite{TGep}.
We hope that the arguments in this note lead to 
future developments of the study of Chern characters of 
stable objects on 3-folds.

\begin{figure}[htbp]
 \begin{center}
  \includegraphics[width=60mm]{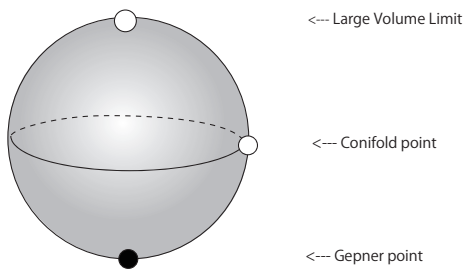}
 \end{center}
 \caption{Stringy K\"ahler moduli space
 of a quintic 3-fold}
 \label{fig:one}
\end{figure}

\subsection{Acknowledgment}
The author would like to thank 
Kentaro Hori, Kyoji Saito and Atsushi 
Takahashi for valuable discussions. 
The author also would like to thank Johannes Walcher
for pointing out the reference~\cite{Wal}. 
This work is supported by World Premier 
International Research Center Initiative
(WPI initiative), MEXT, Japan. This work is also supported by Grant-in Aid
for Scientific Research grant (22684002)
from the Ministry of Education, Culture,
Sports, Science and Technology, Japan.

\subsection{Notation and convention}
All the varieties or polynomials are defined over
complex numbers. For a smooth projective variety $X$
of dimension $n$
and $E \in \Coh(X)$, we write its
Chern character as a vector
\begin{align*}
\ch(E)=(\ch_0(E), \ch_1(E), \cdots, \ch_n(E))
\end{align*}
for $\ch_i(E) \in H^{2i}(X)$. 
For a triangulated category $\dD$ and a set of 
objects $\sS$ in $\dD$, we denote by
$\langle \sS \rangle_{\rm{ex}}$ the smallest
extension closed subcategory in $\dD$
which contains $\sS$. 

\section{Background}\label{sec:back}

\subsection{Bridgeland stability condition}
Let $\dD$ be a triangulated category and $K(\dD)$ its 
Grothendieck group. 
We first recall Bridgeland's definition of 
stability conditions on it. 
\begin{defi}\label{defi:stab} \emph{(\cite{Brs1})}
A stability condition $\sigma$
on $\dD$ consists of a pair $(Z, 
\{\pP(\phi)\}_{\phi \in \mathbb{R}})$
\begin{align}\label{pair2}
Z \colon K(\dD) \to \mathbb{C}, \quad 
\pP(\phi) \subset \dD
\end{align}
where $Z$
is a group homomorphism 
(called central charge) 
and $\pP(\phi)$ is
a full subcategory (called $\sigma$-semistable objects 
with phase $\phi$)
satisfying the following conditions: 
\begin{itemize}
\item For $0\neq E \in \pP(\phi)$, 
we have $Z(E) \in \mathbb{R}_{>0} \exp(\sqrt{-1} \pi \phi)$. 
\item For all $\phi \in \mathbb{R}$, we have 
$\pP(\phi+1)=\pP(\phi)[1]$. 
\item For $\phi_1>\phi_2$ and $E_i \in \pP(\phi_i)$, we have 
$\Hom(E_1, E_2)=0$. 

\item For each $0\neq E \in \dD$, there is
 a collection of distinguished triangles 
\begin{align*}
E_{i-1} \to E_i \to F_i \to E_{i-1}[1], \quad 
E_N=E, \ E_0=0
\end{align*}
with $F_i \in \pP(\phi_i)$ and  
$\phi_1> \phi_2> \cdots > \phi_N$. 
\end{itemize}
\end{defi}
The full subcategory $\pP(\phi) \subset \dD$ is 
shown to be an abelian category, and its 
simple objects are called $\sigma$-stable. 
In~\cite{Brs1}, Bridgeland 
shows that there is a natural topology on 
the  
 set of `good' stability conditions 
$\Stab(\dD)$, and 
its each connected component 
has a structure of a complex manifold. 
Let $\Aut(\dD)$
be the group of autoequivalences on $\dD$. 
There is a left $\Aut(\dD)$-action on 
the set of stability conditions on $\dD$. 
For $\Phi \in \Aut(\dD)$, it
acts on the pair (\ref{pair2}) as 
follows:  
\begin{align*}
\Phi_{\ast} (Z, \{\pP(\phi)\}_{\phi \in \mathbb{R}})
=(Z \circ \Phi^{-1}, \{ \Phi(\pP(\phi)) \}_{\phi \in \mathbb{R}}). 
\end{align*}
There is also a right $\mathbb{C}$-action on 
the set of stability conditions on $\dD$. 
For $\lambda \in \mathbb{C}$, it acts on the pair (\ref{pair2})
as follows: 
\begin{align*}
 (Z, \{\pP(\phi)\}_{\phi \in \mathbb{R}}) \cdot (\lambda)
= (e^{-\sqrt{-1}\pi \lambda} Z, \{ \pP(\phi + \Ree \lambda) \}_{\phi \in \mathbb{R}}). 
\end{align*}
The notion of Gepner type stability conditions is 
defined as follows: 
\begin{defi}\emph{(\cite{TGep})}
A stability condition $\sigma$ on $\dD$ is called 
Gepner type with respect to $(\Phi, \lambda) \in \Aut(\dD) \times \mathbb{C}$
if the following condition holds: 
\begin{align*}
\Phi_{\ast}\sigma= \sigma \cdot (\lambda). 
\end{align*}
\end{defi}

\subsection{Gepner type stability conditions on  
graded matrix factorizations}
Let $W$ be a homogeneous element 
\begin{align}\label{def:A}
W \in A \cneq \mathbb{C}[x_1, x_2, \cdots, x_n]
\end{align}
of degree $d$
such that
$(W=0) \subset \mathbb{C}^n$ has an isolated 
singularity at the origin. 
For a graded $A$-module $P$, 
we denote by $P_i$ its degree $i$-part, 
and $P(k)$ the graded $A$-module 
whose grade is shifted by $k$, i.e. 
$P(k)_{i}= P_{i+k}$. 
\begin{defi}
A graded matrix factorization of $W$ is data
\begin{align}\label{MF}
P^0 \stackrel{p^0}{\to} P^1 \stackrel{p^1}{\to}
P^0(d)
\end{align}
where $P^i$ are graded free $A$-modules of finite rank, $p^i$ 
are homomorphisms of graded $A$-modules, satisfying the 
following conditions:
\begin{align*}
\quad p^1 \circ p^0= \cdot W, \quad 
p^0(d) \circ p^1= \cdot W.
\end{align*}
\end{defi}
The category $\HMF(W)$ 
is defined to be the triangulated 
category whose objects consist of 
graded matrix factorizations of $W$ 
(cf.~\cite{Orsin}). 
 The grade shift functor
$P^{\bullet} \mapsto P^{\bullet}(1)$ 
induces the 
autoequivalence $\tau$ of $\HMF(W)$, 
which satisfies the 
 following identity: 
\begin{align}\label{taud}
\tau^{\times d} =[2]. 
\end{align}
The following is the main conjecture in~\cite{TGep}:
\begin{conj}\label{conj:main}
There is a Gepner type stability condition
\begin{align*}
\sigma_G=(Z_G, \{\pP_G(\phi)\}_{\phi \in \mathbb{R}})
\in  \Stab(\HMF(W))
\end{align*}
with respect to 
$(\tau, 2/d)$, whose central charge $Z_G$ is given by 
\begin{align}\label{Z_G}
Z_G(P^{\bullet})= \mathrm{str}(e^{2\pi \sqrt{-1}/d} \colon 
P^{\bullet} \to P^{\bullet}). 
\end{align}
\end{conj}
The definition of the
central charge $Z_G$ 
first appeared in~\cite{Wal}. 
It
is more precisely 
written as follows: 
since $P^i$ are free $A$-modules of finite rank, 
they are written as 
\begin{align*}
P^i \cong 
\bigoplus_{j=1}^{m} A(n_{j}^{i}), 
\quad n_{j}^{i} \in \mathbb{Z}.
\end{align*}
Then (\ref{Z_G}) is written as 
\begin{align*}
Z_G(P^{\bullet})=
\sum_{j=1}^{m} \left( e^{2 n_{j}^{0} \pi \sqrt{-1}/d} - e^{2 n_{j}^{1} 
\pi \sqrt{-1} /d} \right). 
\end{align*}
So far 
Conjecture~\ref{conj:main} is proved 
when 
$n=1$~\cite{Tak}, $d<n=3$~\cite{KST1}, 
and $n\le d \le 4$~\cite{TGep}.
The most important unproven 
case is when $n=d=5$,
in which the variety $X$ is a quintic Calabi-Yau 3-fold.

\subsection{Orlov's theorem}
We recall Orlov's theorem~\cite{Orsin}
relating 
the triangulated category $\HMF(W)$ with 
the derived category of coherent sheaves on the
smooth projective variety
\begin{align}\label{DMW}
X \cneq (W=0) \subset \mathbb{P}^{n-1}. 
\end{align}
We only use the results
for $d=n$ case, i.e. 
$X$ is a Calabi-Yau manifold, 
and $d=n+1$ case, i.e. $X$ is general type.  
\begin{thm}\emph{(\cite[Theorem~2.5]{Orsin},~\cite[Proposition~5.8]{BFK})}\label{thm:Orlov}
If $d=n$, there is 
an equivalence of triangulated categories
\begin{align*}
\Psi \colon D^b \Coh(X) \stackrel{\sim}{\to} \HMF(W)
\end{align*}
such that the following diagram commutes: 
\begin{align*}
\xymatrix{
D^b \Coh(X) \ar[r]^{\Psi} \ar[d]_{F} & \HMF(W) \ar[d]^{\tau} \\
D^b \Coh(X) \ar[r]^{\Psi} & \HMF(W). 
}
\end{align*}
Here $F$ is the autoequivalence given by 
$F=\mathrm{ST}_{\oO_X} \circ \otimes \oO_X(1)$. 
\end{thm}
Recall that $\mathrm{ST}_{\oO_X}$ is the 
Seidel-Thomas twist~\cite{ST}, given by
\begin{align*}
\mathrm{ST}_{\oO_X}(\ast) =
\mathrm{Cone}(\dR \Hom(\oO_X, \ast) \otimes \oO_X \to 
\ast). 
\end{align*}

\begin{thm}\emph{(\cite[Theorem~2.5]{Orsin},~\cite[Proposition~3.22]{TGep})}\label{thm:Orlov2}
If $d=n+1$, then there is a 
fully faithful functor
\begin{align*}
\Psi \colon D^b \Coh(X) \hookrightarrow \HMF(W)
\end{align*}
such that we have the semiorthogonal decomposition
\begin{align*}
\HMF(W)= \langle \mathbb{C}(0), \Psi D^b \Coh(X) \rangle
\end{align*}
where $\mathbb{C}(0)$ is a certain exceptional object. 
Moreover the subcategory
\begin{align*}
\aA_W \cneq \langle \mathbb{C}(0), \Psi \Coh(X) \rangle_{\rm{ex}}
\end{align*} 
is the heart of a bounded t-structure on $\HMF(W)$,
and there is an equivalence of abelian categories
\begin{align*}
\Theta \colon \mathrm{Syst}(X) \stackrel{\sim}{\to} \aA_W. 
\end{align*}
Here $\mathrm{Syst}(X)$ is the abelian category of coherent 
systems on $X$. 
\end{thm}
Recall that a coherent system 
on $X$ consists of data 
\begin{align*}
V\otimes \oO_X \stackrel{s}{\to} F
\end{align*}
where $V$ is a finite dimensional $\mathbb{C}$-vector 
space, $F \in \Coh(X)$ and 
$s$ is a morphism in $\Coh(X)$.
The set of morphisms in $\mathrm{Syst}(X)$
is given by the 
commutative diagrams in $\Coh(X)$
\begin{align*}
\xymatrix{
V\otimes \oO_X \ar[r]^{s} \ar[d] & F \ar[d] \\
V'\otimes \oO_X \ar[r]^{s'} & F'. 
}
\end{align*}
The equivalence $\Theta$ sends
$(\oO_X \to 0)$ to $\mathbb{C}(0)$ and 
$(0 \to F)$ for 
$F \in \Coh(X)$ to $\Psi(F) \in \aA_W$.

\section{Stronger BG inequality for quintic 3-folds}
In this section, we take $W$ to be a quintic 
homogeneous polynomial with five variables
\begin{align}\label{poly:W}
W \in \mathbb{C}[x_0, x_1, x_2, x_3, x_4], \quad \deg(W)=5. 
\end{align}
The variety 
\begin{align*}
X \cneq (W=0) \subset \mathbb{P}^4
\end{align*}
 is a smooth 
quintic Calabi-Yau 3-fold. 
This is the most interesting case in the study of 
Conjecture~\ref{conj:main}.
We have an equivalence by Theorem~\ref{thm:Orlov}
\begin{align}\label{Or:quin}
\Psi \colon D^b \Coh(X) \stackrel{\sim}{\to} \HMF(W). 
\end{align}
The goal of this section is to 
translate Conjecture~\ref{conj:main} 
in terms of $D^b \Coh(X)$, and relate it to a stronger 
version of BG inequality for stable 
sheaves on $X$.

\subsection{Stringy K\"ahler moduli space of a quintic 3-fold}
Let us first
recall a 
mirror family of a quintic 3-fold $X$ and 
its 
stringy K\"ahler moduli space. 
The mirror family of $X$ is 
a simultaneous crepant resolution $\widehat{Y}_{\psi}
\to Y_{\psi}$
of the 
following one parameter family of 
quotient varieties~\cite{COGP}: 
\begin{align*}
Y_{\psi} \cneq \left\{ \sum_{i=0}^{5}
y_i^5 - 5 \psi \prod_{i=0}^5 y_i =0  \right\} / 
G.  
\end{align*}
Here $[y_1 \colon y_2 \colon y_3 \colon y_4 \colon y_5]$ is the homogeneous 
coordinate of $\mathbb{P}^4$, and 
$G=(\mathbb{Z}/5\mathbb{Z})^3$
acts on $\mathbb{P}^4$ by 
\begin{align*}
\xi \cdot [y_1 \colon y_2 \colon y_3 \colon y_4 \colon y_5]
= [\xi_1 y_1 \colon \xi_2 y_2 \colon \xi_3 y_3 \colon
 \xi_1^{-1} \xi_2^{-1} \xi_3^{-1} y_4 \colon 1] 
\end{align*}
for $\xi=(\xi_i)_{1\le i\le 3} \in G$.
Let $\alpha$ be the root of unity
\begin{align*}
\alpha \cneq e^{2\pi \sqrt{-1}/5}.
\end{align*}
 Note that we have the isomorphism
\begin{align}\label{isom:mirror}
\widehat{Y}_{\psi} \stackrel{\cong}{\to}
\widehat{Y}_{\alpha \psi}
\end{align}
by $y_i \mapsto y_i$ for $1\le i\le 4$
and $y_5 \mapsto \alpha y_5$. 
Also $\widehat{Y}_{\psi}$ is a non-singular 
Calabi-Yau 3-fold if and only if $\psi^5 \neq 1$.
Hence the mirror family $\widehat{Y}_{\psi}$ is parametrized
by the following quotient stack  
(see Figure~\ref{fig:one}) 
\begin{align*}
\mM_{K} \cneq \left[ \frac{\{ \psi \in \mathbb{C} :
\psi^5 \neq 1 \} }{\mu_5}  \right]
\end{align*}
where the generator of $\mu_5$
acts on $\mathbb{C}$ by the multiplication of 
$\alpha$. 
The stack $\mM_K$ is called the 
stringy K$\ddot{\rm{a}}$hler moduli 
space of $X$. 
We see that there 
 are 3-special points in 
Figure~\ref{fig:one}:
\begin{itemize}
\item The point $\psi^5=\infty$, called
\textit{
Large volume limit}. 
\item The point $\psi^5=1$, 
called \textit{Conifold point}. 
\item The point $\psi^5=0$, 
called \textit{Gepner point}.  
\end{itemize}
The mirror variety
$\widehat{Y}_{\psi}$
is non-singular 
except at the first two special 
points. It is also non-singular 
at the Gepner point, but 
there admits a non-trivial
$\mathbb{Z}/5\mathbb{Z}$-action
 by the isomorphism (\ref{isom:mirror}).

\subsection{Relation to Bridgeland stability}
We discuss a relationship between 
the space 
$\mM_K$ 
and the Bridgeland's space 
\begin{align*}
\Stab(X) \cneq \Stab(D^b \Coh(X))
\end{align*}
based on the papers~\cite{Asp2}, \cite{Brs6}. 
Let $\Auteq(X)$ be
the group of autoequivalences of $D^b \Coh(X)$. 
It is expected that there is
an embedding
\begin{align}\label{I}
I \colon \mM_K \hookrightarrow
\left[ \Auteq(X) \backslash
\Stab(X) / \mathbb{C} \right] 
\end{align}
such that, if we write 
\begin{align*}
I(\psi)=(Z_{\psi}, \{\pP_{\psi}(\phi)\}_{\phi \in \mathbb{R}})
\end{align*}
then 
the central charge $Z_{\psi}(E)$
for $E \in D^b \Coh(X)$ 
is a solution of the Picard-Fuchs (PF) equation which
the period integrals of the mirror family 
$\widehat{Y}_{\psi}$ should satisfy. 
Using the following notation
\begin{align*}
z \cneq 5^{-5} \psi^{-5}, \quad 
\theta_z \cneq z \frac{d}{dz}
\end{align*} the PF 
equation is given by 
\begin{align}\label{PF}
\theta_z^4 \Phi  -5z (5\theta_z +1)(5\theta_z +2)(5\theta_z +3)(5\theta_z +4)
\Phi =0. 
\end{align}
The solution space of the above PF equation is 
known to be four dimensional. 
In the $\psi$-variable, the basis is given by
(cf.~\cite{COGP})
\begin{align*}
\varpi_j(\psi) \cneq -\frac{1}{5}
\sum_{m=1}^{\infty}
\frac{\Gamma(m/5)}{\Gamma(m)\Gamma(1-m/5)^4}(5\alpha^{2+j}\psi)^m
\end{align*}
for $0\le j\le 3$. 
For an object $E \in D^b \Coh(X)$, the 
central charge $Z_{\psi}(E)$ 
should satisfy the above PF equation, hence 
is written as 
\begin{align*}
Z_{\psi}(E)= \sum_{i=0}^{3} \Phi_i(\psi) \cdot H^{3-i}\ch_i(E)
\end{align*}
where $H \cneq c_1(\oO_X(1))$ and 
$\Phi_i(\psi)$ is a linear combination of the basis 
$\{\varpi_j(\psi)\}_{0\le j\le 1}$
which is independent of $E$. 
Here we have identified $H^6(X, \mathbb{Q})$ with $\mathbb{Q}$ via
the integration map. 
On the other hand, around the large volume limit and the 
conifold point, the monodromy transformations 
induce linear isomorphisms $M_L$, $M_C$
on the solution space of the PF equation (\ref{PF}). 
Hence that monodromy
transformations act on the central charge
$Z_{\psi}(E)$, 
which are expected to coincide with 
the actions of autoequivalences 
$\otimes \oO_X(1)$, $\mathrm{ST}_{\oO_X}$ respectively. 
Namely we should have the following identities:
\begin{align*}
&Z_{\psi}(E\otimes \oO_X(1))=
\sum_{i=0}^3 M_L \Phi_i(\psi) \cdot H^{3-i}\ch_i(E) \\
&Z_{\psi}(\mathrm{ST}_{\oO_X}(E))=
\sum_{i=0}^3 M_C \Phi_i(\psi) \cdot H^{3-i} \ch_i(E). 
\end{align*}
The coefficients of $\Phi_i(\psi)$ are uniquely 
determined by the above matching property of the 
monodromy transformations on both sides of (\ref{I}). 

Indeed, the above idea is used to 
give an embedding similar to (\ref{I}) 
when $X$ is the local projective plane in~\cite{BaMa}.
In the quintic 3-fold case, based on a similar idea as above, 
the central 
charges $Z_{\psi}(E)$ for line bundles $E=\oO_X(m)$
are computed by Aspinwall~\cite[Equation~(217)]{Asp2}:
\begin{align*}
Z_{\psi}(\oO_X(m))=&
\frac{1}{6}(5m^3 + 3m^2 + 16m +6) \varpi_0(\psi) \\
&-\frac{1}{2}(3m^2 + 3m+ 2) \varpi_1(\psi) -m^2 \varpi_2(\psi)
 -\frac{1}{2} m(m-1)
\varpi_3(\psi). 
\end{align*}
Since $e^{mH}$ for $m\in \mathbb{Z}$ span $H^{\rm{even}}(X, \mathbb{Q})$, the 
above formula uniquely determines $\Phi_i(\psi)$.
A direct computation shows that 
\begin{align*}
\Phi_0(\psi)&=\frac{1}{5}(\varpi_0(\psi)-\varpi_0(\psi)) \\
\Phi_1(\psi)&=\frac{1}{30}(16\varpi_0(\psi)-9\varpi_1(\psi)+ 
3\varpi_3(\psi)) \\
\Phi_2(\psi)&= \frac{1}{5}(\varpi_0(\psi)-3\varpi_1(\psi)-2\varpi_2(\psi)-\varpi_3(\psi)) \\
\Phi_3(\psi)&=\varpi_0(\psi). 
\end{align*}
As a 
result, $Z_{\psi}(E)$ is written as
\begin{align}\notag
&(\varpi_0(\psi)-\varpi_1(\psi)) \ch_0(E) 
+ \frac{1}{30} \left( 16 \varpi_0(\psi) -9\varpi_1(\psi)
+ 3 \varpi_3(\psi) \right) H^2\ch_1(E) \\
\notag
&+ \frac{1}{5}\left(\varpi_0(\psi) - 3\varpi_1(\psi) -2\varpi_2(\psi) - 
\varpi_3(\psi)  \right) 
H \ch_2(E) + \varpi_0(\psi) \ch_3(E). 
\end{align}

\subsection{Gepner point and Gepner type stability conditions}
Let us consider 
a conjectural stability condition
$\sigma_G \in \Stab(X)$ satisfying 
\begin{align*}
[\sigma_G]=
I(\psi^5=0)
 \in \left[ \Auteq(X) \backslash \Stab(X)/ \mathbb{C} \right]
\end{align*}
where $I$ is an expected embedding (\ref{I}).
Since the point 
$\psi^5 =0$ (Gepner point) in $\mM_K$ is an orbifold point 
with stabilizer group $\mathbb{Z}/5\mathbb{Z}$, 
the stability condition $\sigma_G$ should also 
have the stabilizer group $\mathbb{Z}/5\mathbb{Z}$
with respect to the
$\Auteq(X) \times \mathbb{C}$ action on $\Stab(X)$. 
Under a suitable choice of $\sigma_G$, 
the generator of the above stabilizer group
should be given by
\begin{align}\label{ST:stab}
\left( \mathrm{ST}_{\oO_X} \circ \otimes \oO_X(1), -\frac{2}{5} \right)
 \in \Auteq(X) \times \mathbb{C}
\end{align}
since the action of 
$\mathrm{ST}_{\oO_X} \circ \otimes \oO_X(1)$
on $H^{\rm{even}}(X, \mathbb{Q})$
corresponds to the 
composition of monodromy transformations
 at the large volume limit and the 
conifold point under the embedding (\ref{I}), and 
the five times composition of
$\mathrm{ST}_{\oO_X} \circ \otimes \oO_X(1)$
coincides with $[2]$. (This is a consequence of Theorem~\ref{thm:Orlov} 
and the identity (\ref{taud}).)
The property of $\sigma_G$ fixed by (\ref{ST:stab})
is nothing but the Gepner type property with 
respect to $(\mathrm{ST}_{\oO_X} \circ \otimes \oO_X(1), 2/5)$. 
By the above argument and Theorem~\ref{thm:Orlov}, 
a stability condition corresponding to the Gepner point
gives a desired stability condition in 
Conjecture~\ref{conj:main}
via Orlov equivalence (\ref{Or:quin}).  

As for the central charge at the Gepner point, we 
consider the normalized central charge $Z_G^{\dag}$
so that $Z_G^{\dag}(\oO_x)=-1$ holds for any $x\in X$. 
Under this normalization, $Z_G^{\dag}$ is given by 
\begin{align*}
Z_G^{\dag}(E) \cneq & \lim_{\psi \to 0} 
-Z_{\psi}(E)/ \varpi_0(\psi) \\
=& -\ch_3(E) + \frac{1}{5}(\alpha^3 + 2\alpha^2 + 3\alpha -1) H \ch_2(E) \\
& + \frac{1}{30}(-3\alpha^3 +9\alpha -16) H^2 \ch_1(E) + (\alpha-1) \ch_0(E).
\end{align*}
Indeed, the coefficients $\alpha_j^{\dag} \in \mathbb{C}H^{3-j}$ of 
$Z_G^{\dag}(E)$ at 
$\ch_j(E)$ are checked to form the unique solution of the linear equation 
\begin{align*}
(\alpha_0^{\dag}, \cdots, \alpha_3^{\dag})
\cdot M = \alpha \cdot (\alpha_0^{\dag}, \cdots, \alpha_3^{\dag}), \quad 
\alpha_3^{\dag}=-1
\end{align*}
where $M$ is given by the composition of matrices
(cf.~\cite[Subsection~4.1]{TGep})
\begin{align}\notag
M\cneq  
\left( \begin{array}{cccc}
1 & -(\mathrm{td}_X)_2 & 0 & -1 \\
0 & 1 & 0 & 0 \\
0 & 0 & 1 & 0 \\
0 & 0 & 0 & 1
\end{array} \right)
\left( \begin{array}{cccc}
1 & 0& 0 & 0 \\
H & 1 & 0 & 0 \\
H^2/2 & H & 1 & 0 \\
H^3/6 & H^2/2 & H & 1
\end{array} \right). 
\end{align}
Here $(\mathrm{\td}_X)_2=5H^2/6$ is the 
$H^{2, 2}(X)$-component of $\td_X$. 
The above matrix $M$
induces the isomorphism on $H^{\rm{even}}(X)$, which 
is identified 
with the action of $\mathrm{ST}_{\oO_X} \circ \otimes \oO_X(1)$
on it. By~\cite[Proposition~4.4]{TGep}, 
the central charge $Z_G^{\dag}$ is related to 
the central charge $Z_G$ on $\HMF(W)$ given by (\ref{Z_G}) 
as 
\begin{align*}
Z_G(\Psi(E))= -(1-\alpha)^4 \cdot Z_G^{\dag}(E)
\end{align*}
for any $E \in D^b \Coh(X)$. 
Here $\Psi$ is the equivalence (\ref{Or:quin}). 
By 
applying $\mathbb{C}$-action on $\Stab(X)$, 
Conjecture~\ref{conj:main} 
for the polynomial (\ref{poly:W})
leads to the following conjecture: 
\begin{conj}\label{conj:quintic}
Let $X \subset \mathbb{P}^4_{\mathbb{C}}$ be a smooth quintic 3-fold, 
$H\cneq c_1(\oO_X(1))$ and $\alpha \cneq e^{2\pi \sqrt{-1}/5}$. 
Then there is a Gepner type stability condition 
\begin{align}\label{Gep:quintic}
(Z_G^{\dag}, \{\pP_G^{\dag}(\phi)\}_{\phi \in \mathbb{R}}) \in \Stab(X)
\end{align}
with respect to $(\mathrm{ST}_{\oO_X} \circ \otimes \oO_X(1), 2/5)$, 
whose central charge $Z_G^{\dag}$ is given by 
\begin{align*}
Z_G^{\dag}(E) 
=& -\ch_3(E) + \frac{1}{5}(\alpha^3 + 2\alpha^2 + 3\alpha -1) H \ch_2(E) \\
& + \frac{1}{30}(-3\alpha^3 +9\alpha -16) H^2 \ch_1(E) + (\alpha-1) \ch_0(E).
\end{align*}
\end{conj}

\subsection{Some observations}
Let us try to construct a desired stability condition in 
Conjecture~\ref{conj:quintic}. 
By~\cite[Proposition~5.3]{Brs1}, 
giving data (\ref{Gep:quintic}) is equivalent to 
giving the heart of a bounded t-structure 
\begin{align}\notag
\aA_G \subset D^b \Coh(X)
\end{align}
 satisfying 
\begin{align}\label{positivity}
Z_G^{\dag}(\aA_G \setminus \{0\}) 
\subset \{ r \exp(\sqrt{-1} \pi \phi) :
r>0, \phi \in (0, 1] \}
\end{align}
and any object $E \in \aA_G$ admits 
a Harder-Narasimhan filtration with respect to the
$Z_G^{\dag}$-stability. 
We propose that
a desired heart $\aA_G$ is constructed as 
a double tilting of $\Coh(X)$, similar to the one 
in~\cite{BMT}. 
This is motivated by the following observations:

Firstly in~\cite{TGep}, 
we constructed a Gepner type stability condition
for a quartic K3 surface $S$
via a tilting of $\Coh(S)$. 
The construction is similar to the 
one near the large volume limit in~\cite{Brs2}, \cite{AB}.
A different point is that, although 
we only need a classical BG inequality 
to construct a stability condition
near the
large volume limit, 
a construction at the Gepner point requires 
a stronger version of BG inequality
given as follows: 
\begin{lem}
Let $S$ be a K3 surface and $E$ 
a torsion free stable sheaf $E$ on $S$
with $\rank(E) \ge 2$.
Then 
we have the following inequality
\begin{align}\label{strong:K3}
\frac{\Delta(E)}{\rank(E)^2} 
\ge 2- \frac{2}{\rank(E)^2} \ge \frac{3}{2}. 
\end{align}
\end{lem}
The above lemma is an easy 
 consequence of the 
Riemann-Roch theorem and Serre duality
(cf.~\cite[Corollary~2.5]{Mu2})
and a similar improvement is not 
known for other surfaces
except Del Pezzo surfaces. 
By the above observation, we expect that 
 a desired Gepner type stability condition on a quintic 3-fold
is also constructed 
in a way similar to the one near the large volume limit, after an 
an improvement of BG inequality. 

Secondly we can rewrite the central charge $Z_G^{\dag}(E)$
in the following way: 
\begin{align}\label{rewrite}
-\ch_3^{B}(E) &+ a H^2 \ch_1^B(E) + \sqrt{-1} \left( b H \ch_2^B(E) + c \ch_0^{B}(E) \right). 
\end{align}
Here $B=-H/2$ and $\ch^{B}(E)$ is the twisted 
Chern character
\begin{align*}
\ch^B(E) \cneq e^{-B} \ch(E). 
\end{align*}
In (\ref{rewrite}), 
$a, b, c$ are some real numbers in $\mathbb{Q}(\alpha, \sqrt{-1})$, given by 
\begin{align*}
a&= -\frac{1}{5} \alpha^3 -\frac{1}{5} \alpha^2 -\frac{67}{120} \\
b\sqrt{-1}&= \frac{1}{5} \alpha^3 + \frac{2}{5} \alpha^2 + \frac{3}{5} \alpha
+ \frac{3}{10} \\
c\sqrt{-1}&=\frac{3}{8} \alpha^3 + \frac{1}{4} \alpha^2 + \frac{5}{8} \alpha 
+ \frac{5}{16}. 
\end{align*}
They are 
approximated by 
\begin{align*}
a= -0.8819 \cdots, \ b= 0.68819 \cdots, \ 
c= 0.52088\cdots. 
\end{align*}
The expression (\ref{rewrite}) is very similar to the 
central charge near the large volume limit, 
given by 
\begin{align}\notag
Z_{B, tH}(E) \cneq -\int_X e^{-\sqrt{-1} tH} \ch^{B}(E)
\end{align}
for $t\in \mathbb{R}_{>0}$. 
The above integration is expanded as 
\begin{align}\label{expand}
-\ch_3^B(E) + \frac{t^2}{2} H^2 \ch_1^B(E) 
+\sqrt{-1} \left( tH \ch_2^B(E) - \frac{5t^3}{6} \ch_0^B(E) \right). 
\end{align} 
By comparing (\ref{rewrite}) with (\ref{expand}), 
although they are in a similar form, 
we see that some signs of the coefficients 
are different. 
In~\cite{BMT}, we constructed 
a double tilting of $\Coh(X)$ 
which, together with the central charge (\ref{expand}), 
conjecturally gives a Bridgeland stability condition near 
the large volume limit. 
We propose to construct the heart $\aA_G$
via a double tilting of $\Coh(X)$
in a way similar to~\cite{BMT}, by taking the 
difference of the signs of the coefficients into consideration.

\subsection{Conjectural stronger Bogomolov-Gieseker inequality}
We imitate the argument in~\cite{BMT} to construct $\aA_G$. 
In what follows, we fix $B=-H/2$. 
Let $\mu_{B, H}$ be the twisted slope function on $\Coh(X)$ 
defined by
\begin{align*}
\mu_{B, H}(E) \cneq \frac{H^2\ch_1^B(E)}{\rank(E)}. 
\end{align*}
Here we set $\mu_{B, H}(E)=\infty$ if $E$ is a torsion sheaf. 
The above slope function defines the classical slope stability on 
$\Coh(X)$. 
We define the pair of full subcategories 
$(\tT_{B, H}, \fF_{B, H})$ of $\Coh(X)$ to be
\begin{align*}
\tT_{B, H} &\cneq \langle E : \mu_{B, H} \mbox{-semistable with }
\mu_{B, H}(E)>0 \rangle_{\rm{ex}} \\
\fF_{B, H} &\cneq \langle E : \mu_{B, H} \mbox{-semistable with }
\mu_{B, H}(E) \le 0 \rangle_{\rm{ex}}. 
\end{align*}
The above subcategories form a torsion pair in $\Coh(X)$. 
The associated tilting $\bB_{B, H}$ is defined to be
\begin{align*}
\bB_{B, H} \cneq \langle \fF_{B, H}[1], \tT_{B, H} \rangle_{\rm{ex}}. 
\end{align*}
The category $\bB_{B, H}$ is the heart of a bounded t-structure on 
$D^b \Coh(X)$. In~\cite[Lemma~3.2.1]{BMT}, it is 
observed that the central charge (\ref{expand})
satisfies the following condition:
an object $E \in \bB_{B, H}$
with $H^2 \ch_1^B(E)=0$
satisfies
$\Imm Z_{B, tH}(E) \ge 0$. 
The classical BG inequality is used to show 
the above property. 
We propose that a similar property also
holds for the central charge $Z_G^{\dag}$, i.e.
 an object $E \in \bB_{B, H}$ with 
$H^2 \ch_1^B(E)=0$ satisfies
$\Imm Z_G^{\dag}(E) \ge 0$. 
Note 
that such an object $E$ is contained in the category
\begin{align*}
\langle F[1], \Coh_{\le 1}(X) : F \mbox{ is } \mu_{B, H} \mbox{-stable with }
H^2\ch_1^B(F)=0 \rangle_{\rm{ex}}
\end{align*}
where $\Coh_{\le 1}(X)$ is the category of coherent sheaves 
$T \in \Coh(X)$ with $\dim \Supp(T) \le 1$. 
Also noting 
the equality
\begin{align*}
\Delta(E)= \ch_1^B(E)^2 -2\ch_0^B(E) \ch_2^B(E)
\end{align*}
the above requirement leads to the following 
conjecture: 
\begin{conj}\label{strong}
Let $X \subset \mathbb{P}^4$ be a smooth 
quintic 3-fold and $E$ a torsion free slope stable 
sheaf on $X$ 
with $c_1(E)/ \rank(E)=-H/2$.
Then we have the following inequality: 
\begin{align}\label{bound}
\frac{\Delta(E) \cdot H}{\rank(E)^2} >\frac{2c}{b} = 1.5139 \cdots. 
\end{align}
\end{conj}
The RHS of (\ref{bound}) is irrational, hence 
the equality is not achieved. 
Note that the RHS in (\ref{bound}) is very 
close to the RHS in (\ref{strong:K3})
for the K3 surface case. 

\begin{rmk}
A stronger BG inequality similar 
to (\ref{bound}) is predicted by~\cite{DRY}
without the condition $c_1(E)/\rank(E)=-H/2$. 
The prediction in~\cite{DRY}
is shown to be false in~\cite{Jar},~\cite{Naka}. 
Conjecture~\ref{strong} does not contradict to 
the results in~\cite{Jar},~\cite{Naka} since 
we restrict to the sheaves with 
fixed slope $c_1(E)/\rank(E)=-H/2$. 
\end{rmk}
There are few examples of stable sheaves on quintic 3-folds
in literatures. 
The following example is taken in~\cite{Jar}:
\begin{exam}
Let $E$ be the kernel of the morphism 
$\oO_X^{\oplus 6} \to \oO_X(1)^{\oplus 2}$
given by the matrix
\begin{align*}
\left( \begin{array}{cccccc}
x_0 & x_1 & 0 & x_2 & x_3 & 0 \\
0 & x_0 & x_1 & 0 & x_2 & x_3
\end{array}  \right). 
\end{align*}
Here $[x_0: x_1: x_2: x_3: x_4]$ is the homogeneous 
coordinates in $\mathbb{P}^4$. 
By~\cite{Jar}, $E$ is a stable vector bundle on $X$
with 
\begin{align*}
\ch(E)=(4, -2H, -H^2, -H^3/3). 
\end{align*}
Then we have
\begin{align*}
\frac{\Delta(E) \cdot H}{\rank(E)^2}
= \frac{15}{4} > 1.5139 \cdots. 
\end{align*}
\end{exam}
The rank two case will be 
treated in Subsection~\ref{rank:two}.

\subsection{Conjectural construction of the Gepner point}
We now give a conjectural construction of 
a desired $\aA_G$ assuming Conjecture~\ref{strong}.
Similarly to~\cite[Lemma~3.2.1]{BMT}, we have the following lemma: 
\begin{lem}\label{lem:property}
Suppose that Conjecture~\ref{strong} is true. 
Then for any non-zero $E \in \bB_{B, H}$, we have 
the following: 
\begin{itemize}
\item We have $H^2 \ch_1^B(E) \ge 0$. 
\item If $H^2 \ch_1^B(E)=0$, then we have 
$\Imm Z_{G}^{\dag}(E) \ge 0$. 
\item If $H^2 \ch_1^B(E)=\Imm Z_G^{\dag}(E)=0$, 
then $-\Ree Z_{G}^{\dag}(E)>0$. 
\end{itemize}
\end{lem} 
\begin{proof}
The same argument of~\cite[Lemma~3.2.1]{BMT} is applied 
by using Conjecture~\ref{strong} instead of 
the classical BG inequality. 
\end{proof}
The above lemma shows that the triple
\begin{align*}
(H^2 \ch_1^B(E), \Imm Z_G^{\dag}(E), -\Ree Z_G^{\dag}(E))
\end{align*}
should behave like $(\rank, c_1, \ch_2)$ on
coherent sheaves on algebraic surfaces. 
Similarly to the slope function on 
coherent sheaves, we consider the following slope function on 
$\bB_{B, H}$
\begin{align*}
\nu_{G}(E) \cneq 
\frac{\Imm Z_G^{\dag}(E)}{H^2 \ch_1^B(E)}. 
\end{align*}
Here we set $\mu_{G}(E)=\infty$
if $H^2 \ch_1^B(E)=0$. 
If we assume Conjecture~\ref{strong}, 
then Lemma~\ref{lem:property} shows that the slope function $\nu_G$
satisfies the weak see-saw property. 
\begin{defi}
An object $E \in \bB_{B, H}$ is $\nu_G$-(semi)stable if, for any 
non-zero proper subobject $F \subset E$ in $\bB_{B, H}$, we have 
the inequality
\begin{align*}
\nu_{B, H}(F)<(\le) \nu_{B, H}(E/F). 
\end{align*}
\end{defi}
We have the following lemma: 
\begin{lem}
Suppose that Conjecture~\ref{strong} is true. 
Then the $\nu_G$-stability on $\bB_{B, H}$ satisfies the 
Harder-Narasimhan property. 
\end{lem}
\begin{proof}
Although the central charge $Z_G^{\dag}(\ast)$ is 
irrational, the values $H^2 \ch_1^B(\ast)$ are 
contained in $\frac{1}{2} + \mathbb{Z}$, hence they are discrete. 
This is enough to apply the same argument 
of~\cite[Lemma~3.2.4]{BMT},~\cite[Proposition~7.1]{Brs2}
to show the existence of Harder-Narasimhan filtrations 
with respect to $\nu_G$-stability. 
\end{proof}
Assuming Conjecture~\ref{strong}, we 
define the full subcategories in $\bB_{B, H}$
\begin{align*}
\tT_{G} &\cneq \langle E : \nu_{G} \mbox{-semistable with }
\nu_{G}(E)>0 \rangle_{\rm{ex}} \\
\fF_{G} &\cneq \langle E : \nu_{G} \mbox{-semistable with }
\nu_{G}(E) \le 0 \rangle_{\rm{ex}}. 
\end{align*}
As before, the pair $(\tT_G, \fF_G)$
forms a torsion pair on $\bB_{B, H}$. 
By taking the tilting, we obtain the 
heart of a bounded t-structure
\begin{align*}
\aA_G \cneq \langle \fF_G[1], \tT_G \rangle_{\rm{ex}}. 
\end{align*}
We propose the following conjecture: 
\begin{conj}\label{conj:const}
Let $X \subset \mathbb{P}^4$ be a smooth quintic 3-fold
and assume that Conjecture~\ref{strong} is true. 
Then the pair
\begin{align}\label{pair:Gep}
(Z_G^{\dag}, \aA_G)
\end{align}
determines a Gepner type stability condition 
on $D^b \Coh(X)$
with respect to $(\mathrm{ST}_{\oO_X} \circ \otimes \oO_X(1), 2/5)$. 
\end{conj}
\begin{rmk}
By the construction
and the irrationality of $Z_G^{\dag}$,  
the pair (\ref{pair:Gep}) satisfies the 
condition (\ref{positivity}). 
On the other hand, the irrationality of $Z_G^{\dag}$ makes 
it hard to prove the Harder-Narasimhan property
of the pair (\ref{pair:Gep}). 
\end{rmk}

\subsection{Conjecture~\ref{strong} for the rank two case}\label{rank:two}
We show that Conjecture~\ref{strong}
is true in the rank two case. 
\begin{prop}
Conjecture~\ref{strong} is true when $\rank(E)=2$. 
\end{prop}
\begin{proof}
Since we have the inequality
\begin{align*}
\Delta(E^{\vee \vee}) \cdot H \ge \Delta(E) \cdot H
\end{align*}
we may assume that $E$ is reflexive. 
Since $\rank(E)=2$, we have $c_1(E)=-H$
and 
\begin{align*}
\Delta(E) \cdot H= -H^3 + 4c_2(E) \cdot H. 
\end{align*}
The classical BG inequality implies that 
$\Delta(E) \cdot H \ge 0$, i.e. 
$c_2(E) \cdot H \ge 5/4$. 
The conjectural inequality (\ref{bound}) 
is equivalent to that $c_2(E) \cdot H > 2.7639 \cdots$. 
It is enough to exclude the case 
$c_2(E) \cdot H=2$, or equivalently 
$\ch_2(E) \cdot H= 1/2$. 

Suppose by contradiction that $\ch_2(E) \cdot H = 1/2$. 
Let us set $F \cneq E^{\vee}$, which is also a
torsion free slope stable sheaf. 
Since $F$ is reflexive,  
we have 
\begin{align*}
\eE xt^i(F, \oO_X)=0, \quad i\ge 2
\end{align*}
and $Q \cneq \eE xt^1(F, \oO_X)$ is a zero dimensional sheaf
by~\cite[Proposition~1.1.10]{Hu}. 
This implies that there is a distinguished triangle
\begin{align*}
F^{\vee} \to \mathbb{D}(F) \to Q[-1]
\end{align*}
where $\mathbb{D}(\ast)$ is the derived dual 
$\dR \hH om(\ast, \oO_X)$. 
Therefore if we write
\begin{align}\label{chF}
\ch(F)&=(2, H, \ch_2(F), \ch_3(F))
\end{align}
then we have $\ch_2(F^{\vee})=\ch_2(F)=\ch_2(E)$ 
and 
\begin{align}\label{chF2}
\ch(F^{\vee})=(2, -H, \ch_2(E), -\ch_3(F) + \lvert Q \rvert). 
\end{align}
Here 
$\lvert Q \rvert$ is the length of
the zero dimensional sheaf $Q$. 
On the other hand, since $F$ is a rank two
reflexive sheaf, 
we have the isomorphism
(cf.~\cite[Proposition~1.10]{Hart})
\begin{align*}
F \cong F^{\vee} \otimes \det(F). 
\end{align*}
Noting that $\det (F)=\oO_X(H)$, and (\ref{chF}), (\ref{chF2}), 
we have
\begin{align*}
(2, H, \ch_2(E), \ch_3(F))= e^H \cdot 
(2, -H, \ch_2(E), -\ch_3(F) + \vert Q \rvert). 
\end{align*}
The above equality and the assumption 
$\ch_2(E) \cdot H=1/2$ imply that
\begin{align}\label{ch_3}
\ch_3(F)= -\frac{1}{6} + \frac{\lvert Q \rvert}{2}. 
\end{align}
Noting that $c_2(X)=10H^2$, 
the Riemann-Roch theorem and (\ref{ch_3}) imply that
\begin{align}\notag
\chi(F) &\cneq \sum_{i=0}^{3} (-1)^i \dim H^i(X, F) \\
\label{Q}
&= 4 + \frac{\lvert Q \vert}{2}. 
\end{align}
We divide into two cases:
\begin{case}
$H^0(X, F)=0$. 
\end{case}
By the Serre duality and stability, we have
\begin{align*}
H^3(X, F) \cong H^0(F, \oO_X) \cong 0. 
\end{align*}
Therefore, by the assumption $H^0(X, F)=0$ and (\ref{Q}), 
we have 
\begin{align}\label{H2(XF)}
\dim \Ext^1(F, \oO_X)= 
\dim H^2(X, F) \ge 4.
\end{align}
Let us take the universal extension
\begin{align*}
0 \to \oO_X \otimes \Ext^1(F, \oO_X)^{\vee} \to 
\uU \to F \to 0. 
\end{align*}
Then by~\cite[Lemma~2.1]{noteBG}, 
the sheaf $\uU$ is a torsion free slope stable sheaf. 
Applying the BG inequality to $\uU$, we obtain the 
inequality
\begin{align*}
(H^2-2\ch_2(E)(2+ \dim \Ext^1(F, \oO_X))) \cdot H \ge 0. 
\end{align*}
The above inequality implies that 
$\dim \Ext^1(F, \oO_X) \le 3$, which contradicts to 
(\ref{H2(XF)}).

\begin{case}
$H^0(X, F) \neq 0$. 
\end{case}
Let us take a non-zero element $s \in H^0(X, F)$, and 
an exact sequence
\begin{align}\label{OFM}
0 \to \oO_X \stackrel{s}{\to} F \to M \to 0. 
\end{align}
By~\cite[Lemma~2.2]{noteBG}, 
the sheaf $M$ is a torsion free slope 
stable sheaf. 
Therefore it is written as 
\begin{align*}
M \cong \oO_X(H) \otimes I_Z
\end{align*}
for some subscheme $Z \subset X$
with $\dim Z \le 1$. 
We have the equalities of Chern characters
\begin{align*}
\ch_2(F) &= \frac{1}{2}H^2 - [Z] \\
\ch_3(F) &= \frac{1}{6}H^3 - H \cdot [Z] -\chi(\oO_Z). 
\end{align*}
Because $\ch_2(F) \cdot H= \ch_2(E) \cdot H=1/2$, 
we have $H \cdot [Z]=2$. Hence we obtain
\begin{align*}
\ch_3(F)= -\frac{7}{6} -\chi(\oO_Z). 
\end{align*}
On the other hand, (\ref{ch_3}) implies that $\ch_3(F) \ge -1/6$, 
hence we have $\chi(\oO_Z) \le -1$. 
By taking the generic 
projection of the one dimensional subscheme $Z \subset \mathbb{P}^4$
to $\mathbb{P}^3$, 
the 
Castelnuovo inequality implies 
\begin{align*}
g(Z) \cneq h^1(\oO_Z) \le \frac{1}{2}(H \cdot [Z] -1) (H \cdot [Z]-2). 
\end{align*}
Since $H \cdot [Z]=2$, we have $h^1(\oO_Z)=0$, which 
contradicts to $\chi(\oO_Z) \le -1$. 
\end{proof}

\section{Clifford type bound for quintic surfaces}\label{sec:Cli}
In this section, we take $W'$ to be a quintic homogeneous 
polynomial with four variables
\begin{align*}
W' \in \mathbb{C}[x_0, x_1, x_2, x_3], \quad \deg(W')=5.
\end{align*}
We consider Conjecture~\ref{conj:main} in this case. 
We relate it with some Clifford type bound 
for stable coherent systems on 
the smooth quintic surface
\begin{align*}
S\cneq (W'=0) \subset \mathbb{P}^3. 
\end{align*}
\subsection{Computation of the central charge} 
The surface $S$ is a hyperplane section
$(x_4=0)$ of a 
quintic 3-fold $X \cneq (W=0) \subset \mathbb{P}^4$, where 
$W$ is defined by
\begin{align*}
W \cneq W' + x_4^5 \in \mathbb{C}[x_0, x_1, x_2, x_3, x_4]. 
\end{align*} 
By Theorem~\ref{thm:Orlov2}, there is the heart of 
a bounded t-structure 
$\aA_{W'} \subset \HMF(W')$,
and an equivalence
\begin{align*}
\Theta \colon \mathrm{Syst}(S) \stackrel{\sim}{\to} 
\aA_{W'}. 
\end{align*}
Below we abbreviate $\Theta$ and regard
a coherent system $(\oO_S^{\oplus R} \to F)$
as an object in $\aA_{W'}$. 
There is a natural push-forward functor (cf.~\cite{UedaM})
\begin{align*}
i_{\ast} \colon 
\HMF(W') \to \HMF(W) 
\end{align*}
such that by~\cite[Lemma~3.12]{TGep}
and~\cite[Lemma~4.5]{TGep}, 
we have 
\begin{align*}
i_{\ast}(\oO_S^{\oplus R} \to F) \cong \Psi(\oO_X^{\oplus R} \to i_{\ast}F). 
\end{align*}
Here 
$i_{\ast}F$ is the usual sheaf push-forward for the 
embedding $i \colon S \hookrightarrow X$, 
$\Psi \colon D^b \Coh(X) \stackrel{\sim}{\to}
\HMF(W)$ an equivalence in Theorem~\ref{thm:Orlov} 
and 
\begin{align*}
(\oO_X^{\oplus R} \to i_{\ast}F) \in D^b \Coh(X)
\end{align*}
is an object in the derived category
with 
$i_{\ast}F$ located in degree zero. 
Let us consider the central charge $Z_G^{'\dag}$ 
on $\HMF(W')$ defined by 
\begin{align*}
Z_G^{'\dag}(P) \cneq Z_G^{\dag}(\Psi^{-1} i_{\ast} P), \quad P \in \HMF(W')
\end{align*}
where $Z_G^{\dag}$ is the central charge (\ref{rewrite})
on $D^b \Coh(X)$
considered in the previous section. 
By the argument in~\cite[Section~4]{TGep}, 
the central charge $Z_G^{'\dag}$
on $\HMF(W')$ differs from 
(\ref{Z_G}) only up to a scalar multiplication. 
For $F \in \Coh(S)$, let us write
\begin{align*}
\ch(F)=(r, l, n) \in H^0(S) \oplus H^2(S) \oplus H^4(S)
\end{align*}
with $r\in \mathbb{Z}$ and $n\in \frac{1}{2} + \mathbb{Z}$. 
By setting $H=c_1(\oO_X(1))$ and $B=-H/2$, 
we have 
\begin{align*}
&\ch^B(\Psi^{-1}i_{\ast}(\oO_S^{\oplus R} \to F)) \\
&= \ch^B(\oO_X^{\oplus R} \to i_{\ast}F) \\
&=\left( -R, \left( r- \frac{R}{2} \right) H, 
i_{\ast}l -\frac{R}{8}H^2, 
n+ \frac{5}{24}r - \frac{5}{48}R \right). 
\end{align*}
Applying the computation of $Z_G^{\dag}$ in the previous section, 
we have
\begin{align*}
Z_G^{'\dag}(\oO_S^{\oplus R} \to F) 
= -n-\frac{5}{24}r &+ \frac{5}{48}R + 5a\left( r-\frac{R}{2} \right) \\
&+ \sqrt{-1} \left( b\left( h \cdot l - \frac{5}{8}R \right) -cR  \right). 
\end{align*}
Here $h\cneq H|_{S}$ and 
$a, b, c$ are irrational numbers given in (\ref{rewrite}).

\subsection{Conjectural Clifford type bound}
We expect that a desired Gepner type 
stability condition in this case is constructed 
via double tilting of $\aA_{W'}$, similarly to the 
previous section. 
Let $\mu'$ be the slope function on $\aA_{W'}$, given by 
(using the notation in the previous subsection)
\begin{align*}
\mu'(\oO_S^{\oplus R} \to F) &\cneq -\frac{\ch_1^B(i_{\ast}F) \cdot H^2}{R} \\
&=5 \left(\frac{1}{2}- \frac{\rank(F)}{R} \right). 
\end{align*}
Here we set $\mu'(\ast)=-\infty$ if 
$R=0$. 
(Also see~\cite[Subsection~5.4]{TGep}.)
The above slope function defines the $\mu'$-stability 
on $\aA_{W'}$, which 
satisfies the Harder-Narasimhan property (cf.~\cite[Lemma~5.14]{TGep}). 
Following the same argument in the previous section, 
we expect that any $\mu'$-stable
 object $E \in \aA_{W'}$ with 
$\mu'(E)=0$ satisfies 
$\Imm Z_G^{'\dag}(E) \ge 0$. 
It leads to the following conjecture: 
\begin{conj}\label{conj:Clifford}
Let $S \subset \mathbb{P}^3$ be a 
smooth quintic surface and
$h=c_1(\oO_S(1))$. For a $\mu'$-stable 
coherent system  
$(\oO_S^{\oplus R} \to F)$ on $S$ with 
$R= 2 \rank(F)>0$, we have the following inequality
\begin{align*}
\frac{c_1(F) \cdot h}{R} > \frac{5}{8} + \frac{c}{b} =1.3818 \cdots. 
\end{align*}
\end{conj}
If we assume the above conjecture, we are able to 
construct a double tilting $\aA_{G}'$ of $\aA_{W'}$, 
such that the pair $(Z_G^{'\dag}, \aA_G')$
satisfies
\begin{align*}
Z_G^{'\dag}(\aA_G' \setminus \{0\})
\subset \{ r \exp(\sqrt{-1} \pi \phi) : r>0, \phi \in (0, 1] \}. 
\end{align*}
We conjecture that the pair $(Z_G^{'\dag}, \aA_G')$
gives a Gepner type stability 
condition on $\HMF(W')$ with respect to $(\tau, 2/5)$. 
The construction of $\aA_G'$ is similar to $\aA_G$ in the 
previous section, and we 
leave the readers to give its explicit construction. 
We just check the easiest case of Conjecture~\ref{conj:Clifford}:
\begin{lem}
Conjecture~\ref{conj:Clifford} is true if 
$R=2\rank(F)=2$. 
\end{lem} 
\begin{proof}
Let $(\oO_S^{\oplus 2} \stackrel{s}{\to} F)$ be a $\mu'$-stable 
coherent system on $S$ with $\rank(F)=1$. 
The inequality in Conjecture~\ref{conj:Clifford}
is equivalent to $c_1(F) \cdot h > 2.7636 \cdots$. 
It is enough to show that $c_1(F) \cdot h \ge 3$. 
Let $F \twoheadrightarrow F'$ be a torsion free quotient. 
There is a surjection in $\aA_{W'}$
\begin{align*}
(\oO_S^{\oplus 2} \to F) \twoheadrightarrow 
(\oO_S^{\oplus 2} \to F')
\end{align*}
whose kernel is of the form $(0 \to F'')$
for a torsion sheaf $F''$ on $S$. 
Obviously $(\oO_S^{\oplus 2} \to F')$ is also 
$\mu'$-stable, and $c_1(F') \cdot h \le c_1(F) \cdot h$. 
Hence we may assume that $F$ is torsion free. 
Also note that $h^0(F) \ge 2$, since otherwise 
there is an injection in $\aA_{W'}$
\begin{align*}
 (\oO_S \to 0) \hookrightarrow (\oO_S^{\oplus 2} \to F)
\end{align*}
satisfying
\begin{align*}
\mu'(\oO_S \to 0)=5/2>0=\mu'(\oO_S^{\oplus 2} \to F)
\end{align*} 
which contradicts to the $\mu'$-stability of $(\oO_S^{\oplus 2} \to F)$. 
Let us set $\lL \cneq F^{\vee \vee}$, and take 
a smooth member $C \in \lvert h \rvert$. 
Note that $\lL$ is a line bundle satisfying 
$h^0(\lL)\ge 2$, and $C$ is a smooth quintic curve in $\mathbb{P}^2$. 
Suppose by contradiction that 
$c_1(F) \cdot h =c_1(\lL) \cdot h \le 2$. 
We have the exact sequence
\begin{align*}
0 \to \lL(-C) \to \lL \to \lL|_{C} \to 0. 
\end{align*}
Since $c_1(\lL(-C)) \cdot h= c_1(\lL) \cdot h -5 <0$
by our assumption, we have $h^0(\lL(-C))=0$
and $h^0(\lL|_{C}) \ge 2$. On the other hand, 
Clifford's theorem on $C$ yields (cf.~\cite[Theorem~5.4]{Ha2})
\begin{align*}
h^0(\lL|_{C}) \le \frac{1}{2} \deg (\lL|_{C}) + 1 \le 2. 
\end{align*}
Furthermore, the first inequality is strict 
since $\lL|_{C} \neq 0, K_C$, and $C$ is not hyperelliptic. 
Therefore we obtain a contradiction. 
\end{proof}

\providecommand{\bysame}{\leavevmode\hbox to3em{\hrulefill}\thinspace}
\providecommand{\MR}{\relax\ifhmode\unskip\space\fi MR }
\providecommand{\MRhref}[2]{%
  \href{http://www.ams.org/mathscinet-getitem?mr=#1}{#2}
}
\providecommand{\href}[2]{#2}

Kavli Institute for the Physics and 
Mathematics of the Universe, University of Tokyo,
5-1-5 Kashiwanoha, Kashiwa, 277-8583, Japan.

\textit{E-mail address}: yukinobu.toda@ipmu.jp

\end{document}